\documentclass{amsart}

\usepackage{amsmath,amsfonts,amssymb}
\usepackage{graphicx,psfrag}

\usepackage{tikz-cd}
\usepackage{mathtools}

\usepackage{import}
\usepackage{xifthen}
\usepackage{pdfpages}
\usepackage{transparent}

\usepackage{amsthm}\theoremstyle{definition}
\newtheorem{defi}{Definition}
\newtheorem{teo}[defi]{Theorem}
\newtheorem{lemma}[defi]{Lemma}
\newtheorem{clly}[defi]{Corollary}

\newcommand{\R}{{\mathbb{R}}}

\newcommand{\C}{{\mathbb{C}}}
\newcommand{\Z}{{\mathbb{Z}}}
\newcommand{\N}{{\mathbb{N}}}

\DeclareMathOperator{\id}{Id}

\begin{document}

\title{ QUOTIENTS OF TORUS ENDOMORPHISMS HAVE PARABOLIC ORBIFOLDS}
\author{ SOF\'IA LLAVAYOL AND JULIANA XAVIER}

\address{ Instituto de Matem\'atica y Estad\'istica ``Rafael Laguardia'', Facultad de Ingenier\'{\i}a,  Universidad de la Rep\'ublica, Montevideo, Uruguay.}
\email{jxavier@fing.edu.uy}
\email{sllavayol@fing.edu..uy}

\begin{abstract}
    In this work we show that every quotient of a torus endomorphism has a parabolic orbifold, answering a question of Mario Bonk and Daniel Meyer posed in \cite{BM}.
\end{abstract}
\maketitle

\section{Introduction}

If $f:S^2\to S^2$ is a branched covering map, we denote by $\deg(f)$ its degree and by $\deg(f,p)$ its local degree at $p\in S^2$. The set of \textit{critical points} of $f$ is 
$S_f= \big\{x:\ \deg(f,x)\geq 2\big\}$, and the set of \textit{postcritical points} is given by $P_f= \bigcup_{k\geq 1} f^k(S_f)$. Of course, $f(P_f)\subset P_f$ and $f^{-1}(P_f)$ 
contains both $S_f$ and $P_f$.\\

 A {\it Thurston map} is an orientation preserving branched covering of the 
sphere onto itself such that the postcritical set $P_f$ is finite. The ramification function $\nu_f$  of a Thurston map $f$ is 
the smallest among functions $\nu:S^2\to \N^*\cup\{\infty\}$ such that \begin{itemize}
                                                                         \item $\nu(x)=1$ if $x\notin P_f$
                                                                         \item $\nu(x)$ is a multiple of $\nu(y)\deg(f,y)$ for each $y\in f^{-1}(x)$.\\

                                                                        \end{itemize}

 The \textit{orbifold} associated to a Thurston map $f$ is the pair $\mathcal{O}_f= (S^2, \nu_f)$, and its \textit{Euler characteristic} is defined as
\[ \chi(\mathcal{O}_f)= 2 -\sum_{p\in P_f} \left(1 -\frac{1}{\nu_f(p)}\right).\] 

This number is non positive (see, for example, \cite{BM} Proposition 2.12). We say that the orbifold is \textit{parabolic} if $\chi(\mathcal{O}_f)= 0$ and \textit{hyperbolic}
if $\chi(\mathcal{O}_f)< 0$. \\

The  parabolic case is rather special, as these maps enjoy some sort of ``combinatoric privilege''.  For instance, they can be lifted to degree $d$ covering maps of either the open 
annulus or the torus. They can even act directly on the open annulus by restriction (for example, when there are exactly two totally invariant critical points).

\bigskip A \textit{quotient of a torus endomorphism} is a map $f:S^2\to S^2$ such that there exists a torus endomorphism $F:T^2\to T^2$ with $\deg(F)\geq 2$ and a branched covering
map $\pi:T^2\to S^2$ such that $\pi F= f\pi$.  We point out that neither
$F$ nor $\pi$ are unique.  We use the term \textit{QOTE} when referring to a quotient of a torus endomorphism. Every such map is a Thurston map (see Lemma 3.12 in \cite{BM} 
for a proof). \\

The so called
{\it Latt\`es maps} are famous examples of parabolic orbifolds that appear as the quotient of a torus endomorphisms. Also, the first example of a Thurston map admitting a 
completely invariant {\it indecomposable} continuum is a QOTE (see \cite{ply}).  This problem of finding an indecomposable
continuum as a completely invariant set of a Thurston map relates to a longstanding open problem in holomorphic dynamics: does there exists a rational function with an
indecomposable continuum as its Julia set?.  So, the example given in \cite{ply} gives a positive answer in the topological context (this example can actually be made smooth, but not 
Thurston equivalent to a rational map). \\

It is natural to ask whether every quotient of a torus endomorphism has a parabolic orbifold. The authors learnt this question from the 
amazing book \cite{BM} (see the remarks after Lemma 3.13). We refer the reader to the introduction in \cite{BM_lattes_type} for an excellent review and context on this problem. In that paper, 
the question was 
answered positively for those maps that admit an expanding behavior, but a general
answer remained an open problem.   \\

The purpose of this work is to solve this problem in full generality.\\

We prove the following:

\begin{teo}\label{teo}  If $f:S^2\to S^2$ is a quotient of a torus endomorphism, then $f$ has a parabolic orbifold. 
 
\end{teo}

In order to do so, we introduce the concept of $\pi$-injectivity (Definition \ref{def1}) which enables us to prove a new parabolicity criterion (Corollary \ref{int}) that we show
to be always verified after taking a suitable quotient (Theorem \ref{coc}).\\

\section{Notations and idea of proof}

For a branched covering map $f:X\to Y$ between compact and connected oriented surfaces, we denote by $\deg(f)$ its topological degree and by $\deg(f,x)$ its local degree at $x\in X$.
The set of all critical points of $f$ is denoted by $S_f$.
If $Z$ is another compact and connected oriented surface, and $f:X\to Y$ and  $g:Y\to Z$ are branched covering maps, then $g f$ is also a branched covering map and
  $\deg(g f,x) =
\deg(g,f(x))\deg(f,x)$ for all $x\in X$.  Also, $\deg (g f)=\deg (g) \deg (f)$.\\

From now on, we will assume $f:S^2\to S^2$ is a QOTE, with its associated maps $F: T^2\to T^2$ and $\pi: T^2 \to S^2$ as in the definition.  Note that $\deg(F,x) =1$ for all 
$x\in T^2$ as $F$ is a local homeomorphism.  We will often refer to $F$ as the {\it associated map} and to $\pi$ as the {\it associated projection}.\\

The branched covering map $\pi:T^2\to S^2$ verifies $\pi(S_\pi)= P_f$ (see \cite{BM} Lemma 3.12). Then, restricting $\pi$ to $\widetilde X= T^2\backslash \pi^{-1}(P_f)$ we obtain a
covering map $\widetilde X\to X= S^2\backslash P_f$. This covering space might not be normal and, in fact, it being normal implies that $f$ has a parabolic orbifold (see \cite{BM}
Lemma 3.13). Setting $Y= X\backslash f^{-1}(P_f)$ and $\widetilde Y= T^2\backslash \pi^{-1}(f^{-1}(P_f))$ again gives a covering space $\widetilde Y\to Y$, and note also that
the restrictions $f|_Y$ and $F|_{\widetilde Y}$ are both covering maps.\\

That is,  we have a commutative diagram of covering spaces:

\begin{center}\begin{tikzcd}
    \widetilde Y \arrow[d, "\pi"'] \arrow[r, "F"] & \widetilde X \arrow[d, "\pi"] \\
    Y \arrow[r, "f"'] & X
\end{tikzcd}\end{center}

Furthermore, for $n\geq 0$ set $Y_n= S^2\backslash f^{-n}(P_f)$ and call $\widetilde Y_n= T^2\backslash \pi^{-1}( f^{-n}(P_f))$. Note that for $n=0,1$ we obtain the spaces $X, \widetilde X, Y$ and 
$\widetilde Y$. \\

\begin{lemma}
    For all $n\geq 0$, $Y_{n+1}\subset Y_n$ and $f(Y_{n+1})= Y_n$.
\end{lemma}

\begin{proof}
    Because $P_f\subset f^{-1}(P_f)$, one gets $f^{-n}(P_f)\subset f^{-(n+1)}(P_f)$. Then we have $S^2\backslash f^{-(n+1)}(P_f) \subset S^2\backslash f^{-n}(P_f)$, which means
    $Y_{n+1}\subset Y_n$. On the other hand,  $y\in Y_{n+1}$ if and only if $f^n(f(y))= f^{n+1}(y)\not\in P_f$ if and only if $f(y)\not\in f^{-n}(P_f)$ if and only if $f(y)\in Y_n$. That is, 
    $f(Y_{n+1})\subset Y_n$.\\
\end{proof}

 One then obtains the commutative diagram:

\begin{center}\begin{tikzcd}
    \widetilde Y_{n+1} \arrow[d, "\pi"'] \arrow[r, "F"] & \widetilde Y_n \arrow[d, "\pi"] \arrow[r, "F"] & \cdots \arrow[r, "F"] & \widetilde Y \arrow[d, "\pi"'] \arrow[r, "F"] & \widetilde X \arrow[d, "\pi"] \\
    Y_{n+1} \arrow[r, "f"'] & Y_n \arrow[r, "f"'] & \cdots \arrow[r, "f"']  & Y \arrow[r, "f"'] & X
\end{tikzcd}\end{center}

\noindent where all maps involved are covering maps.\\

We finish this section explaining the idea behind the proof of Theorem \ref{teo}. \\

The property of $\pi$-injectivity of the map $F$ in Definition \ref{def1} below is the main tool needed in the proof. We prove in Theorem \ref{coc} that given a QOTE $f$, we  can 
factor out 
``superflous'' elements in the associated projection to obtain a new  
projection $\pi$  with associated map $F$ that has the desired $\pi$- injectivity property.  There is a flaw with $\pi$- injectivity: this property does not hold for every point in 
$S^2$. The concept of transversality
of $F$ and $\pi$ is the way to transfer $\pi$- injectivity  to the whole sphere (and not only the points of $Y$) (see Definition \ref{def2} and Lemma \ref{trans} below). Finally, it 
follows easily from the transversality of $F$ and $\pi$ that 
$\deg(\pi, \cdot)$ is constant over 
the fibers $\pi^{-1}(p)$ for all $p\in S^2$ (see Lemma \ref{para} below). Theorem \ref{teo} now follows from Bonk and Meyer's {\it parabolicity criterion}  (Lemma 3.13 in \cite{BM}):\\

\begin{lemma}\label{crit} Let $f$ be a QOTE and $\pi:T^ 2\to S^ 2$ as in the definition.  Then, $f$ has a parabolic orbifold if and only if $$\deg(\pi,x)=\deg(\pi,y)$$\noindent for all
$x,y\in T^2$ with $\pi (x)= \pi (y)$.\\
 
\end{lemma}

\section{$\pi$-injectivity and transversality of $F$ and $\pi$}

\begin{defi}\label {def1} 
 $F$ is \textit{$\pi$-injective} if $F$ is injective when restricted to each set $\pi^{-1}(y)$ with $y\in Y$. 
 
\end{defi} 

That is, for a $\pi$-injective endomorphism, $F(x)= F(y)$ implies $\pi(x)\neq \pi(y)$ for all $x,y\in \widetilde Y$.\\

\begin{defi} \label{def2}

We say that $\pi$ and $F$ are \textit{transverse} if given $x\in S^2$,  $y\in f^{-1}(x)$ and $\widetilde x\in \pi^{-1}(x)$, then there exists $\widetilde y\in \pi^{-1}(y)$ such that 
$F(\widetilde y)= \widetilde x$.
 
\end{defi}

\begin{lemma}\label{trans}
    If $F$ is $\pi$-injective, then $\pi$ and $F$ are transverse.
\end{lemma}

\begin{proof}
    Let $x\in S^2$ and take $y\in f^{-1}(x)$, $\widetilde x\in \pi^{-1}(x)$. We want to find $\widetilde y\in \pi^{-1}(y)\cap F^{-1}(\widetilde x)$. Set $n= \deg \pi$.\\

    Suppose first that $x\in X$, that is, $x$ is a regular value of $\pi$. Then $\pi^{-1}(x)= \big\{\widetilde x= \widetilde x_1, \ldots, \widetilde x_n\big\}$ has $n$ elements in 
    $\widetilde X$. As  $y\in f^{-1}(x)$, then $y\in Y\subset X$ is also a regular value of $\pi$ and $\pi^{-1}(y)= \big\{\widetilde y_1, \ldots, \widetilde y_n\big\}$ has $n$ 
    elements in $\widetilde Y$. Because $\pi F=f\pi$, we also have $F(\pi^{-1}(y)) \subset \pi^{-1}(x)$. By assumption, $F$ maps $\pi^{-1}(y)$ injectively to $\pi^{-1}(x)$. As both 
    sets have $n$ elements, it follows $F$ maps a (unique) point $\widetilde y\in \pi^{-1}(y)$ to $\widetilde x$.\\

 Now suppose that $x$ is a critical value of $\pi$ and take $x_n\to x$ a sequence of regular values of $\pi$. There exists $y_n\to y$ such that $f(y_n)= x_n$ for all $n$. Fix a 
 neighbourhood $\widetilde U$ of $\widetilde x$ such that $\widetilde U\cap \pi^{-1}(x)= \{\widetilde x\}$, and for all $n$ choose $\widetilde x_n\in \pi^{-1}(x_n)\cap \widetilde U$. 
 As each $x_n$ is a regular value of $\pi$, there exists $\widetilde y_n\in \pi^{-1}(y_n)$ such that $F(\widetilde y_n)= \widetilde x_n$. Let $\widetilde y$ be an accumulation point
 of $\widetilde y_n$. Continuity gives $F(\widetilde y)= \widetilde x$ and $\pi(\widetilde y)= y$, as desired.\\
\end{proof}

\begin{lemma}
    If $F$ is $\pi$-injective, then $F^n$ is $\pi$-injective for all $n\geq 1$.
\end{lemma}

\begin{proof}
    We will prove the result by induction on $n$. Then the hypothesis gives the base case. For the inductive step, suppose $F^n$ is $\pi$-injective.  Note that this means that $F^n$ 
    is injective when restricted to each set $\pi^{-1}(y)$ with $y\in Y_n$, as $Y_n= S^2\backslash f^{-n}(P_f).$\\

    The fact that $Y_{n+1}\subset Y$ implies that $F$ is injective over $\pi^{-1}(y)$, $y\in Y_{n+1}$. Also, if $y\in Y_{n+1}$, then $f(y)\in f(Y_{n+1})\subset Y_n$ and so $F^n$ is
    injective in $\pi^{-1}(f(y))$. Then we have the following:
    \[
        \pi^{-1}(y) \xhookrightarrow{F} \pi^{-1}(f(y)) \xhookrightarrow{F^n} \pi^{-1}(f^{n+1}(y)),
    \]
    
   \noindent where the hooked arrows mean that the maps are injective. So $F^{n+1}$ maps $\pi^{-1}(y)$ injectively to $\pi^{-1}(f^{n+1}(y))$ for all $y\in Y_{n+1}$. This finishes the inductive step.\\
\end{proof}

Putting together the two previous lemmas we get:\\

\begin{clly}
    If $F$ is $\pi$-injective, then $\pi$ and $F^n$ are transverse for all $n\geq 1$.\\
\end{clly}

\begin{lemma}\label{para}
    If $F$ is $\pi$-injective, then $\deg(\pi, \cdot)$ is constant over the fibers $\pi^{-1}(p)$ for all $p\in S^2$.
\end{lemma}

\begin{proof}
    Let $\widetilde p$ and $\widetilde p'$ be points in $\pi^{-1}(p)$. Note that $\pi^{-1}(P_f)$ is a finite set and that the numbers $\# F^{-n}(\widetilde p)=
    \# F^{-n}(\widetilde p')\geq 2^n$. Then, we may take  $n$ sufficiently large such that there exists $\widetilde q\in F^{-n}(\widetilde p)$ and $\widetilde q'\in F^{-n}(\widetilde p')$ such 
    that $q= \pi(\widetilde q)$ and $q'= \pi(\widetilde q')$ lie outside $P_f$. Now recall that $\pi(S_{\pi})=P_f$, so $\deg(\pi, \widetilde q)= \deg(\pi, \widetilde q')= 1$.\\
    
    Then, 
    
    \[\begin{split}      
\deg(f^n, q)&= \deg(\pi, \widetilde q)\deg(f^n, q)= \deg( f^n\pi, \widetilde q)\\&
= \deg(\pi F^n , \widetilde q)=\deg(F^n, \widetilde q)\deg(\pi,\widetilde p)= 
    \deg(\pi, \widetilde p).\end{split}\]

    \noindent Analogously, $\deg(f^n, q')= \deg(\pi, \widetilde p')$. By the 
    previous corollary, we know $\pi$ and $F^n$ are transverse. That is,  there exists $\widetilde y\in \pi^{-1}(q')$ such that $F^n(\widetilde y)= \widetilde p$. We point out that the
    crucial point here is that $F^n$ maps $\widetilde y$ to $ \widetilde p$ and not $ \widetilde p'$. Note also
    that again $\deg(\pi, \widetilde y)= 1$ as $q'\not\in P_f$. Then, \\ $$\deg(\pi, \widetilde p)= \deg(\pi F^n,\widetilde y)= \deg(f^n \pi, \widetilde y)= \deg(f^n, q')= 
    \deg(\pi, \widetilde p').$$\\
    \noindent This finishes the proof.\\
\end{proof}

Now, using the parabolicity criterion in Lemma \ref{crit} (Lemma 3.13 in \cite{BM}), we get the following result:\\

\begin{clly}\label{int}
    If $F$ is $\pi$-injective, then $f$ has a parabolic orbifold.
\end{clly}

The question arises if every associated map $F$ of a QOTE $f$ is $\pi$-injective. As we will see next, this
is not always the case.

\section{Example}

Before proceeding any further we describe a simple example in great detail to illustrate the ideas and key concepts on this work. In particular, we construct explicit examples
of  $\pi$-injective  and non $\pi$-injective associated maps.\\

This section is independent of the rest of the paper; the whole proof
of Theorem \ref{teo} is contained in the remaining sections. The reader 
may skip this one if desired.  \\

Let $\Gamma$ be the subgroup of automorphisms of the complex plane $\mathbb{C}$ generated by $z\mapsto z+1$ and $z\mapsto z+i$. Let $\Gamma_0$ be the one generated by $\Gamma$ and
the map $z\mapsto -z$. The space $\mathbb{C}/\Gamma$ is a torus $T^2$ and $\mathbb{C}/\Gamma_0$ is a sphere $S^2$. The quotient projection 
$\pi: \mathbb{C}/\Gamma \to \mathbb{C}/ \Gamma_0$ is a two-fold branched covering of the torus onto the sphere.\\

Let $\widetilde F: \C \to \C$, $\widetilde F(z)=2z$.  Then, $\widetilde F$ projects to $F: T^2 \to T^2$ and to 
$f:S^2\to S^2$. That is, one obtains $\pi F=f\pi$ and the projected map $f$ is a QOTE. Note that $S_\pi=\{\bf{(0,0),(1/2,0), (1/2, 1/2),(0,1/2)}\}$, where the boldface means their
class $\mod \Z^2$. 
As was already pointed out, $P_f=\pi (S_\pi)$. Note  that $\pi(S_\pi)\subset S^2\backslash S_f $; that is, there are no critical points of $f$ inside $P_f$.
 As $\deg(f)=4$, $f$ has $6$ critical points by the Riemann-Hurwitz formula.  Namely, the image by $\pi$ of the points 
 $\{{\bf{(0,1/4),(1/4,0), (1/4, 1/4),(1/2,1/4), (1/4,1/2), (1/4,3/4)}}\}.$ We then have $S_f=f^{-1}(P_f)\backslash P_f$.\\
    
Note that $\pi$ carries each of its critical points to some $x\in S^2$ such that $\#\pi^{-1}(x)=1$.  So, 
if $y\in f^{-1}(x)\cap \pi(S_{\pi})^c$ it is impossible for $F$ to be injective over $\pi^{-1}(y)$, as $\pi^{-1}(y)$ is a set of two points that is mapped to the singleton $\pi^{-1}(x)$.  However, 
as an easy computation shows, if we keep $y\in Y=S^2\backslash f^{-1}(P_f)$ then $F$ is injective over $\pi^{-1}(y)$.  Indeed, the $\pi$- fiber of a point $y\in S^2$ lifts to $\C$ as the 
points $\pm z + (m,n), (m,n)\in \Z^2$, where $\pi({\bf z})=\pi({\bf -z})=y$.  If $F$ is not injective over $\pi^{-1}(y)$, this means that $2z=-2z+(m,n)$ for some integers $m$ and $n$.  Equivalently,
$4z\in \Z^2$ or $z\in \frac{1}{4}\Z^2$.  But then, $\widetilde F(z)=2z \in \frac{1}{2}\Z^2$, and $\frac{1}{2}\Z^2$ is exactly the lift to $\C$ of $S_\pi$, meaning 
${\bf z}\in F^{-1}(S_\pi)$ or equivalently $y\notin Y$. \\

However, even when restricting to the right subset of the sphere $Y$, one can mess with fiber injectivity easily.  For instance, define $\widehat \pi=\pi F$.  Then, 
$f\widehat \pi =f\pi F=\pi F F=\widehat \pi F$, and then $f$ is also a QOTE with projection $\widehat \pi: T^2\to S^2$ and the same map $F:T^2\to T^2$.  Of course by construction, 
$F$ is not 
 injective over the sets $\widehat\pi^{-1}(y), y\in Y$.  These ideas are illustrated in Figure 1. \\
 
This is what motivates Theorem \ref{coc} in the following section.  We need to find an
appropriate pair of endomorphism and projection to sort this ``problem''. \\

\section{Parabolicity}

We show in this section that we  can factor out ``superflous'' elements in the associated projection $\pi$ of a QOTE $f$ to obtain a new 
projection that has the desired fiber injectivity property.\\

Let $G= \big\{\psi:T^2\to T^2:\ F\psi=F\big\}$  be the group of covering transformations of $F$ and let $H= \big\{ \psi\in G:\ \pi\psi= \pi\big\}$. It is not hard to check that $H$
is a subgroup of $G$.

\begin{lemma} If $F$ is not $\pi$-injective, then $H$ is nontrivial.
 
\end{lemma}

\begin{proof} As $F$ is not $\pi$-injective, there exists $\widetilde y_1\neq \widetilde y_2$ in $\widetilde Y$ be such that $\pi(\widetilde y_1)= \pi(\widetilde y_2)= y$ and 
    $F(\widetilde y_1)= F(\widetilde y_2)= \widetilde x$.  Because the fundamental group of $T^2$ is abelian, the covering $F$ is normal. Then, as $\widetilde y_1, \widetilde y_2$ are
    both in the fiber
    $F^{-1}(\widetilde x)$, there exists $\varphi \in G$ satisfying $\varphi(\widetilde y_1)= \widetilde y_2$.  Note that $\varphi\neq \id$ as $\widetilde y_1\neq \widetilde y_2$.\\

    First we claim that $\varphi(\widetilde Y)\subset \widetilde Y$. Indeed, for $\widetilde y\in \widetilde Y$ note that $f\pi \varphi(\widetilde y)= \pi F \varphi(\widetilde y)= \pi F(\widetilde y)\in X$, and so $\pi \varphi(\widetilde y)\in Y$. It follows that $\varphi(\widetilde y)\in \widetilde Y$. This gives a well defined map $\pi \varphi|_{\widetilde Y}: \widetilde Y\to Y$ such that $f \pi\varphi= \pi F\varphi= \pi F$.

    \begin{center}\begin{center}\begin{tikzcd}
        \widetilde Y \arrow[d, "F"'] \arrow[r, "\pi", "\pi\varphi"'] & Y \arrow[d, "f"] \\
        \widetilde X \arrow[r, "\pi"'] & X
    \end{tikzcd}\end{center}\end{center}

    Now, $f:Y\to X$ is a covering map and both $\pi |_{\widetilde Y}$ and $\pi\varphi |_{\widetilde Y}$ are $f$-lifts of $\pi F:\widetilde Y\to X$. What is more, 
    $\pi(\widetilde y_1)= \pi(\widetilde y_2)= \pi\varphi(\widetilde y_1)$. This two conditions give $\pi= \pi\varphi$ on $\widetilde Y$, and then $\pi= \pi\varphi$ on $T^2$.
    That is, $\varphi\in H.$\\
 
\end{proof}

\begin{figure}[h]
    \centering%
    \def\svgwidth{\columnwidth}
    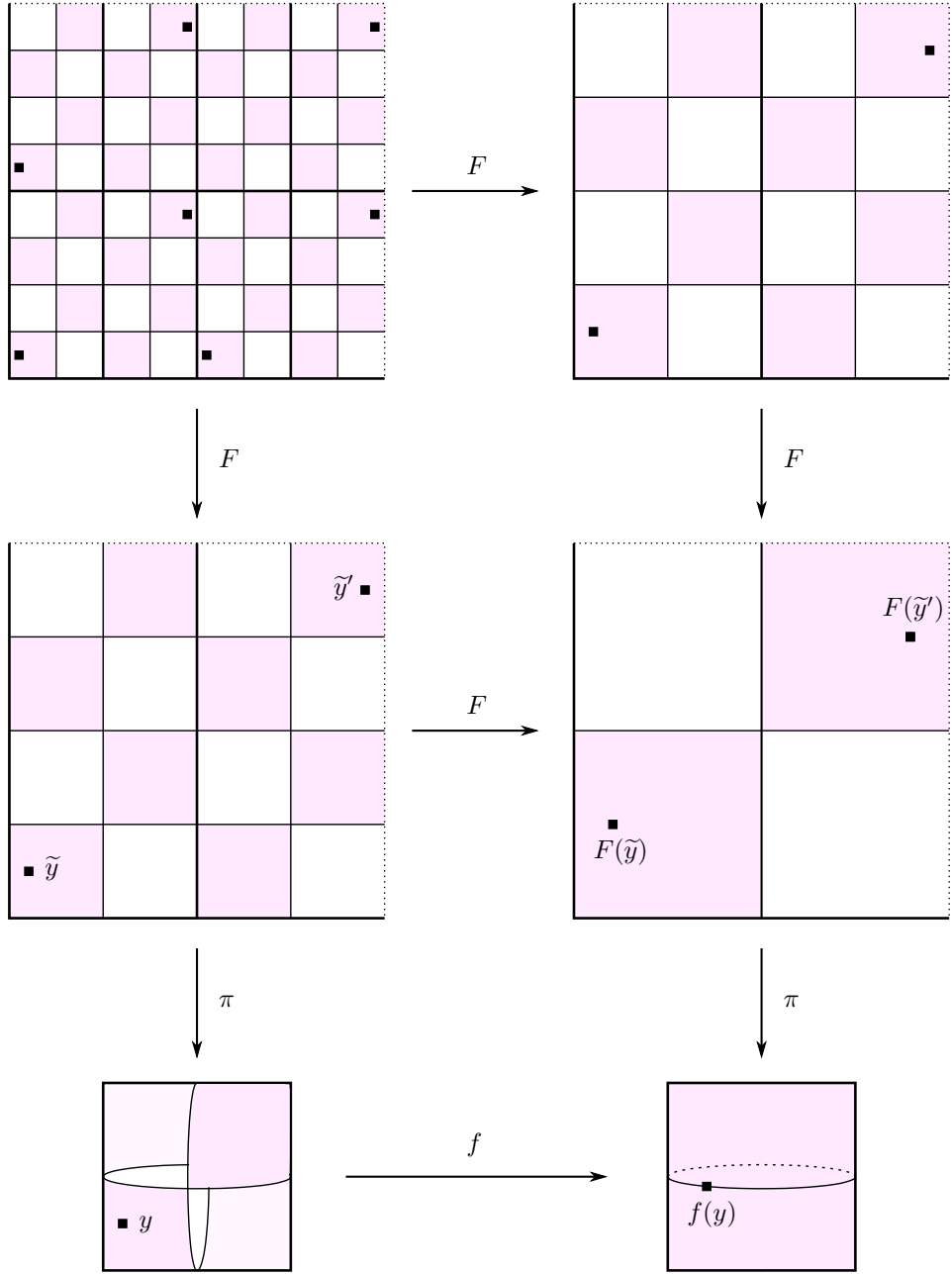

    \caption{Top row: $F$ is not $\pi F$-injective. Middle row: $F$ is $\pi$-injective}
\end{figure}

\begin{teo}\label{coc} There exists a branched covering map $\widehat\pi: T^2\to S^2$ and a torus endomorphism $\widehat F: T^2\to T^2$ of degree $\deg(\widehat F)\geq 2$ such that 
$\widehat\pi \widehat F= f \widehat\pi$ and $\widehat F$ is $\widehat\pi$-injective.
\end{teo}

\begin{proof}  If $F$ is $\pi$-injective, then the result is obvious.  If $F$ is not $\pi$-injective, then by the previous lemma, the subgroup $H$ of $G$ is nontrivial.  Then, we obtain a covering $\overline{\pi}: T^2\to T^2/H$ and $T^2/H$ is (homeomorphic to) a torus by the Riemann-Hurwitz formula. The non triviality of $H$ gives 
     $\deg(\overline{\pi})\geq 2$.\\

    The fibers $\overline{\pi}^{-1}(z)$ contain points that not only project by $\pi$ to a common point, but also their image by $F$ is constant. To check this, take $z\in T^2/H$ and 
    let $\widetilde z_1, \widetilde z_2\in \overline{\pi}^{-1}(z)$. By definition, there exists $\psi\in H$ such that $\psi(\widetilde z_1)= \widetilde z_2$. Because $\psi$ 
    verifies $F\psi= F$ and $\pi\psi= \pi$, the claim follows. Now, $F$ factors to a map $\widehat F: T^2/H\to T^2/H$, that is, $\widehat F \overline{\pi}= \overline{\pi} F$.
    Because $\overline{\pi}$ is a covering map, $\widehat F$ can be written locally as $\overline{\pi} F \overline{\pi}^{-1}$. Then $\widehat F$ is a local homeomorphism from the 
    torus to itself (i.e. a torus endomorphism).\\

    To finish, we define $\widehat\pi: T^2/H\to S^2$ as follows. For $z\in T^2/H$ choose any $\widetilde z\in \overline{\pi}^{-1}(z)$ and define $\widehat\pi(z)= \pi(\widetilde z)$.
    This definition does not depend on the choice of the point in $\overline{\pi}^{-1}(z)$ as was already pointed out. Again, because $\overline{\pi}$ is a covering map,
    $\widehat\pi$ can be written locally as $\pi \overline{\pi}^{-1}$. Then $\widehat\pi$ is a branched covering map from the torus onto the sphere.\\

    Note that $f\widehat\pi\overline{\pi}= f\pi=\pi F= \widehat\pi \overline{\pi} F=\widehat\pi \widehat F\overline{\pi}$ which implies $f \widehat\pi= \widehat\pi \widehat F$ as the 
    map $\overline{\pi}: T^2\to T^2/H$ is surjective. We then have a commutative diagram as below and $\deg(\widehat F)= \deg F = \deg f \geq 2$.

    \begin{center}\begin{tikzcd}
        T^2 \arrow[d, "\overline{\pi}"'] \arrow[dd, bend right=65, "\pi"'] \arrow[r, "F"] & T^2 \arrow[d, "\overline{\pi}"] \arrow[dd, bend left=65, "\pi"] \\
        T^2/H \arrow[d, "\widehat\pi"'] \arrow[r, "\widehat F"] & T^2/H \arrow[d, "\widehat\pi"] \\
        S^2 \arrow[r, "f"'] & S^2
    \end{tikzcd}\end{center}

    What is more, $\deg(\pi)= \deg(\widehat\pi) \deg(\overline{\pi})$ and $\deg(\overline{\pi})\geq 2$. One can conclude $\deg(\pi)> \deg(\widehat\pi)$. In case $\widehat F$ is not
    $\widehat\pi$-injective, we can repeat the process, and it will end since the degree of the projection strictly decreases.\\
\end{proof}

We are now ready to prove Theorem \ref{teo}:\\

\begin{proof}  The previous theorem allows us to suppose that $F$ is $\pi$-injective.  Now, the result follows from Corollary \ref{int}.
 
\end{proof}

\section{Characterization of Lattt\`es-type maps}

In this section we use Theorem \ref{teo} to finish the work started in Chapter 3 in \cite{BM}.\\

We will need some definitions and notations.\\

If $\nu_f$ is the ramification function of a Thurston map $f$, then $\{p\in S^2:\nu_f(p) \geq 2\}=P_f$ is a finite set.  If we label these points
$p_1,\dots,p_n$ such that $2\leq\nu_f(p_1)\leq\ldots \leq \nu_f(p_n)$, then the $n$-tuple $(\nu_f(p_1), \ldots, \nu_f(p_n))$ is called the {\it signature} of 
$(\mathcal{O}_f, \nu_f)$.\\

If $f,g:S^2\to S^2$ are Thurston maps, we say that they are {\it Thurston equivalent} if there exists homeomorphisms $h_0, h_1:S^2\to S^2$ that are isotopic rel. $P_f$ and 
satisfy $gh_1=h_0f$.\\

Let $G$ be a group of homeomorphisms acting on $\C$. The group $G$ acts {\it properly discontinuously} on $\C$ if for each compact set $K\subset \C$ there are only finitely many maps
$g\in G $ such that $g(K)\cap K\neq \emptyset$.\\

We say that $G$ is a {\it crystallographic group} if each element $g\in G$ is an orientation preserving isometry of $\C$ and
if the action of $G$ on $\C$ is properly discontinuous and cocompact.\\

We call a map $A:\R^2\to \R^2$ {\it affine}, if it has the form $A(u)=L_A(u)+u_0, u,u_0\in \R^2$, where $L_A$ is linear.\\

We say that $f:S^2\to S^2$ is a {\it Latt\`es-type map} if there exists a crystallographic group $G$, an affine map $A:\R ^2\to \R ^2$ with $\det(L_A)>1$ that is $G$-equivariant,
and a branched covering map $p:\R^2\to S^2$ induced by $G$ such that $fp=pA$.\\

It is a natural question whether every quotient of a
torus endomorphism $f$ is Thurston equivalent to a Latt\`es-type map.  This question was also asked in \cite{BM}, but Theorem \ref{teo} was needed for a proof:\\

\begin{teo}   The following are equivalent:

\begin{itemize}
 \item [(i)] $f$ is a Thurston map with parabolic orbifold and no periodic critical points.

\item [(ii)] $f$ is Thurston equivalent to a Latt\`es-type map.

\item[(iii)] $f$ is Thurston equivalent to a quotient of a torus endomorphism.
\end{itemize}

\end{teo}

\begin{proof}
  Proposition 3.6 in \cite{BM} shows that (i) and (ii) are equivalent.  By Proposition 3.5 in \cite{BM} every Latt\`es-type map is a quotient of a torus endomorphism, and so (ii)
  implies (iii). To see  (iii) implies (i), we first remark that Thurston equivalent maps have the same signatures (see \cite{BM}, Proposition 2.15).  Then, by Theorem \ref{teo} 
  $f$ has parabolic orbifold because a QOTE does. Morover, if a Thurston map with parabolic orbifold has periodic critical points, then the signature 
  is either
  $(\infty, \infty)$ or $(2,2,\infty)$ (\cite{BM} Propositions 2.9 and 2.14). To finish the proof we recall that a QOTE has no periodic critical points by Lemma 3.12 in \cite{BM}. \\
\end{proof}

\section{Acknowledgments}  Mario Bonk, Daniel Meyer and Rafael Potrie carefully read a preliminary version of this paper and their comments and suggestions improved this work in numerous ways.  We are 
so grateful, not only for this but for their words of encouragement that made us feel over the moon.


\begin{thebibliography}{9}
    \bibitem[BM]{BM}
    M.Bonk, D.Meyer. Expanding Thurston maps. ISBN-10: 0-8218-7554-X ISBN-13: 978-0-8218-7554-4 2017.\\

    \bibitem[BM2]{BM_lattes_type}
    M.Bonk, D.Meyer. Quotients of torus endomorphisms and Lattes-type maps. Arnold Math. J. 6 (2020), no. 3-4, 495–521.\\
    
    
\bibitem[IPRX]{ply} {\ J. Iglesias, A. Portela, A. Rovella, J. Xavier.} {\it Branched coverings of the sphere having
a completely invariant continuum with infinitely many Wada Lakes.}{\ To appear in Topology and its applications.
arXiv:2211.03571}, {2022}.
\end{thebibliography}
\end{document}